\documentclass[a4paper,12pt]{article}
\usepackage[utf8]{inputenc}
\usepackage{amsfonts,amssymb,amsmath,amsthm}
\usepackage{textcomp}
\usepackage[scr=boondoxo]{mathalfa}
\usepackage{enumitem}
\usepackage[all]{xy}
\usepackage{tikz}
\usepackage{mathtools}
\usepackage{color}
\usepackage[unicode]{hyperref}
\usepackage{stmaryrd}
\usepackage{makecell}
\usepackage{booktabs}

\hypersetup{
    colorlinks,
    linkcolor={blue!80!black},
    citecolor={blue!80!black},
    urlcolor={blue!80!black}
}

\DeclareFontFamily{U}{mathx}{\hyphenchar\font45}
\DeclareFontShape{U}{mathx}{m}{n}{
      <5> <6> <7> <8> <9> <10>
      <10.95> <12> <14.4> <17.28> <20.74> <24.88>
      mathx10
      }{}
\DeclareSymbolFont{mathx}{U}{mathx}{m}{n}
\DeclareFontSubstitution{U}{mathx}{m}{n}
\DeclareMathAccent{\widecheck}{0}{mathx}{"71}
\DeclareMathAccent{\wideparen}{0}{mathx}{"75}

\newcommand{\coleq}{\coloneqq}

\newcommand{\bine}{\mathbin{\varepsilon}}
\newcommand{\e}{\varepsilon}

\theoremstyle{plain}
\newtheorem{theorem}{Theorem}
\newtheorem*{definition*}{Definition}
\newtheorem{proposition}[theorem]{Proposition}

\newtheorem{remark}[theorem]{Remark}
\newtheorem*{remark*}{Remark}

\newcommand{\cB}{{\mathcal{B}}}
\newcommand{\cC}{\mathcal{C}}
\newcommand{\cD}{\mathcal{D}}
\newcommand{\cE}{\mathcal{E}}

\newcommand{\cH}{\mathcal{H}}
\newcommand{\cL}{\mathcal{L}}
\newcommand{\cM}{\mathcal{M}}
\newcommand{\cN}{\mathcal{N}}
\newcommand{\cO}{\mathcal{O}}
\newcommand{\cS}{\mathcal{S}}
\newcommand{\bN}{\mathbb{N}}

\newcommand{\abso}[1]{\left|#1\right|}
\newcommand{\norm}[1]{\left\lVert#1\right\rVert}

\title{A survey on duals\\of topological tensor products}
\author{E.~A.~Nigsch\footnote{Wolfgang-Pauli-Institut, Oskar-Morgenstern-Platz 1, 1090 Wien, Austria. e-mail: \href{mailto:eduard.nigsch@univie.ac.at}{eduard.nigsch@univie.ac.at}. Corresponding author.}, N. Ortner\footnote{Institut für Mathematik, Universität Innsbruck, Technikerstraße 13, 6020 Innsbruck, Austria.}}
\date{September 23, 2016}

\begin{document}

\maketitle

\begin{abstract}
We give a survey on classical and recent results on dual spaces of topological tensor products as well as some examples where these are used.
\end{abstract}

{\bfseries Keywords: }topological tensor product, dual space

{\bfseries MSC2010 Classification: 46A32, 46F05}

\section{Introduction}

The (erroneous!) examples
\begin{alignat*}{3}
(\cD'_x \widehat\otimes \cD'_y)'_b & \cong \cD_x \widehat\otimes \cD_y \qquad && \lightning \\
(\cE_x \widehat\otimes \cD'_y)'_b & \cong \varinjlim_{m} ( \cE^{\prime m}_x \widehat\otimes \cD_y)\qquad && \lightning
\end{alignat*}
show that the determination of the strong duals of the space of distributions $\cD'_x \widehat\otimes \cD'_y \cong \cD'_{xy}$ \cite[Prop.~28, p.~98]{zbMATH03145498} or of the space of semiregular distributions $\cE_x \widehat\otimes \cD'_y$ \cite[p.~99]{zbMATH03145498} is far from being trivial.

The first of these isomorphisms would be a consequence of the isomorphism $\cD_x \widehat\otimes \cD_y \cong \cD_{xy}$ wrongly claimed to hold in \cite[p.~500]{Jarchow} by taking into account on the left-hand side the isomorphisms
\[ \cD_{xy} \cong (\cD'_{xy})'_b \cong ( \cD'_x \widehat\otimes \cD'_y )'_b \]
which hold due to reflexivity of $\cD$ and the kernel theorem. Its correct form uses the \emph{inductive} tensor product topology $\iota$ on $\cD_x \otimes \cD_y$ (cf.~\cite[Chap.~II, \S 3, n°3, p.~84]{zbMATH03199982}) which gives
\[ (\cD'_x \widehat\otimes \cD'_y)'_b \cong \cD_x \widehat\otimes_\iota \cD_y. \]

The second isomorphism is given in \cite[Proposition 1, p.~112]{MR0090777} in the wrong form above, and in \cite[Corollary 1, p.~116]{MR0090777} in the corrected form
\[ (\cE_x \widehat\otimes \cD'_y)'_b \cong \varinjlim_{m} ( \cE^{\prime m}_{c,x} \widehat\otimes \cD_y), \]
wherein $\cE^{\prime m}_c$ is the space $\cE^{\prime m}$ endowed with the topology of uniform convergence on compact sets.

These difficulties are further exemplified by A.~Grothendieck's example \cite[Ch.~II, \S 4, n°1, Corollaire 2, p.~98]{zbMATH03199982} of the strong dual of the space of semiregular and semicompact distributions,
\[ (\cE_x \widehat \otimes \cE'_y)'_b \neq \cE'_x \widehat\otimes_\iota \cE_y. \]
In fact, the strong dual of $\cE_x \widehat\otimes \cE_y'$ is not quasi-complete, while the space $\cE'_x \widehat\otimes_\iota \cE_y$ of trace class operators on $\cE$ (cf.~\cite[Chap.~I, \S 3, n°2, D\'efinition 4, p.~80]{zbMATH03199982}) is complete.

Summarizing, one has to be cautious when applying an isomorphism of the form $(E \widehat\otimes F)'_b \cong E'_b \widehat\otimes F'_b$: this isomorphism holds only under certain restrictive conditions while the general situation is much more complicated and delicate to handle, as we will see.

Throughout this article we use the standard notation of \cite{zbMATH03199982,FDVV,zbMATH03145498,zbMATH03145499}.

\section{Classical textbook results}\label{sec_class}

We now give a list of the classical results on duals of tensor products which can be found in the literature. Throughout this section, let $\cH$ and $F$ be locally convex spaces. We follow the usual terminology when we say that a locally convex space is of type (F), (DF), (FM) or (DFM). The topologies of $\cH'_b$, $\cH'_\lambda$ and $\cH'_c$ are those of uniform convergence on bounded, on precompact and on absolutely convex compact subset of $\cH$, respectively.

The first group of results concerns the case where $\cH$ is nuclear and an additional assumption holds for both $\cH$ and $F$:

\begin{tabular}{@{}llll@{}} \toprule
 $\cH$ (nuclear) & $F$ & Duality result & Reference \\
\midrule
(F) & (F) & $(\cH \widehat\otimes F)'_b \cong \cH'_b \widehat\otimes F'_b$ & \cite[Prop. 50.7, p.~524]{Treves}, \\
& & & \cite[9.9 Thm., p.~175]{Schaefer} \\
(F) / (DF) & (F) / (DF) & $(\cH \widehat\otimes F)'_b \cong \cH'_b \widehat\otimes F'_b$ & \cite[Ch.~II, Thm.~12, p.~76]{zbMATH03199982} \\ \bottomrule
\end{tabular}

The second group of results, going back to \cite{zbMATH03386564}, omits the assumption of nuclearity of $\cH$ but in some places employs the $\bine$-product (see \cite{zbMATH03145498}) instead of the completed tensor product. For arbitrary locally convex spaces $\cH$ and $F$, again with symmetric assumptions, we have:

\begin{tabular}{@{}llll@{}} \toprule
 $\cH$ & $F$ & Duality result & Reference \\ \midrule
(F) & (F) & $(\cH \widehat\otimes_\pi F)'_\lambda \cong \cH'_\lambda \bine F'_\lambda$ & \cite[\S 45.3(1), p.~301]{Koethe2} \\
 & & $(\cH \bine F)'_\lambda \cong \cH'_\lambda \widehat\otimes_\pi F'_\lambda$ & \cite[\S 45.3(5), p.~302]{Koethe2} \\
(FM) & (FM) & $(\cH \bine F)'_b \cong \cH'_b \widehat\otimes_\pi F'_b$ & \cite[\S 45.3(7), p.~304]{Koethe2} \\
(DFM) & (DFM) & $(\cH \widehat\otimes_\pi F)'_b \cong \cH'_b \bine F'_b$ & \cite[\S 45.3(7), p.~304]{Koethe2} \\
& & $(\cH \widehat\otimes_\pi F)'_\lambda \cong \cH'_\lambda \bine F'_\lambda$ & \cite[1. Satz, p.~212]{zbMATH03498641} \\ \bottomrule
\end{tabular}

In \cite[16.7, p.~346]{Jarchow} for two Fr\'echet spaces $\cH$, $F$ also the duals of $\cH \widehat\otimes_\pi F$ and $\cH \bine F$ for a topology $\gamma$ defined in \cite[9.3, p.~178]{Jarchow} are determined; moreover, the result of the last line also holds for (DCF)-spaces, see \cite{zbMATH03577774}.

The third group of results, which are less well-known as they do not appear in classical textbooks on locally convex spaces, are given with asymmetric conditions on the spaces $\cH$ and $F$ plus a supplementary condition; in order to keep the following table readable we abbreviate the properties of being \emph{complete, quasi-complete, nuclear} and \emph{semireflexive} by the letters \emph{c,q,n} and \emph{s}, respectively. Concerning the supplementary conditions in the third column, $\beta$-$\beta$ signifies $\beta$-$\beta$-decomposability of the bounded subsets of $\cH \widehat \otimes F$, and analogously for $\gamma$-$\gamma$ (see \cite[p.~15]{zbMATH03145499}).

\begin{tabular}{@{}lllll@{}} \toprule
$\cH$ & $F$ & supp. & Duality result & Reference \\ \midrule
cn & cs & $(\cH \widehat\otimes F)'_b$ c & $(\cH \widehat\otimes F)'_b \cong \cH'_b \widehat\otimes_\iota F'_b$ & \cite[Ch.~II, Corollaire, p.~90]{zbMATH03199982} \\
n & & $\beta$-$\beta$ & $(\cH \widehat\otimes F)'_b \cong \cH'_c ( F'_b; \e)$ & \cite[Prop.~22, p.~103]{zbMATH03145499} \\
nq & q & $\gamma$-$\gamma$ & $(\cH \widehat\otimes F)'_c \cong \cH'_c ( F_c'; \e)$ & \cite[Prop.~22, p.~103]{zbMATH03145499} \\ \bottomrule
\end{tabular}

We remark that in \cite[Prop.~22, p.~103]{zbMATH03145499} no assumptions on the completeness of $\cH$ and $F$ are made; only for the last result we need their quasi-completeness. Furthermore, if one does not have $\beta$-$\beta$ or $\gamma$-$\gamma$-decomposability, the respective isomorphism holds only algebraically.

\section{Examples (I)}

We now give examples where the results of the first and second group above (involving symmetric conditions on $\cH$ and $F$) do not apply and the third group has to be involved. First, by Grothendieck's result we have
\begin{alignat*}{2}
 (\cD' \widehat\otimes \cS)'_b & \cong \cD \widehat\otimes_\iota \cS', & \qquad (\cD' \widehat\otimes \cS')'_b \cong \cD \widehat\otimes_\iota \cS = \overline{\cD}(\cS), \\
 (\cD' \widehat\otimes \cE)'_b & \cong \cD \widehat\otimes_\iota \cE', & \qquad (\cD' \widehat\otimes \cE')'_b \cong \cD \widehat\otimes_\iota \cE = \overline{\cD}(\cE).
\end{alignat*}

(The completeness of $(\cD' \widehat\otimes \cE)'_b$ and $(\cD' \widehat\otimes \cS)'_b$ follows from the bornologicity of $\cD' \widehat\otimes \cE$ and $\cD' \widehat\otimes \cS$.)

Moreover, the following \emph{algebraic} identities hold
\begin{gather*}
 (\cD' \widehat\otimes \cC_0)' = \cD \widehat\otimes \cM^1 = \cL_b ( \cD', \cM^1), \\
(\cD' \widehat\otimes L^1)' = \cD \widehat\otimes L^\infty = \cL_b ( \cD', L^\infty), \\
(\cD' \widehat\otimes \dot\cB^m)' = \cD \widehat\otimes \cD^{\prime m}_{L^1} = \cL_b ( \cD', \cD^{\prime m}_{L^1} ).
\end{gather*}

(Note that $\dot\cB^0 = \cC_0$, $\cD^{\prime 0}_{L^1} = \cM^1$.)

In order to show that these are also topological identities one requires that the bounded subsets of $\cD' \widehat\otimes F$, where $F$ is one of $\cC_0$, $L^1$ or $\dot\cB^m$, are $\beta$-$\beta$-decomposable, which follows from the proof of \cite[Proposition 1, p.~16]{zbMATH03145499}.

\section{The need for generalization}

In the study of convolution of distributions the dual space
\[ (\cD' \widehat\otimes \cD'_{L^1})'_b \cong \cD \widehat\otimes_\iota \cB = \overline{\cD}(\cB) \]
of the space of partially summable distributions is determined \cite[Prop.~3, p.~541]{zbMATH03163214}. This result does not follow from Grothendiecks duality Lemma cited in the third table because $\cD'_{L^1}$ is not semi-reflexive. Duals of tensor products $\cH \widehat\otimes F$ with spaces $F$ being not necessarily semi-reflexive ($\cH$ nuclear) \emph{were recently determined} in \cite{bdot} using the following results, of which the second is a generalization of \cite[Chap.~II, Corollaire, p.~90]{zbMATH03199982}:

\begin{proposition}[{\cite[Proposition 8]{bdot}}]\label{prop8}
 Let $\cH = \varinjlim_k \cH_k$ be the strict inductive limit of nuclear Fr\'echet spaces $\cH_k$ and $F$ be the strong dual of a distinguished Fr\'echet space. Then $(\cH'_b \widehat\otimes F)'_b \cong \overline{\cH}(F'_b) \coleq \varinjlim_k ( \cH_k(F'_b))$. The space $\overline{\cH}(F'_b)$ is a complete, strict (LF)-space and $\cH'_b \widehat\otimes F$ is distinguished. If $F$ is reflexive then $\cH'_b \widehat\otimes F$ is reflexive, too.
\end{proposition}

With the particular space $\cD_{-}$ for $\cH$ and $F$ the strong dual of a \emph{reflexive} Fr\'echet space, Proposition \ref{prop8} can be found in \cite[p.~315]{zbMATH03223921}.

 \begin{proposition}[{\cite[Proposition 9]{bdot}}]\label{prop9}
  Let $\cH$ be a Hausdorff, quasicomplete, nuclear, locally convex space with the strict approximation property, $F$ a quasicomplete, semireflexive, locally convex space. Let $F_0$ be a locally convex space such that $(\cH \widehat\otimes F)'_b = (\cH \widehat \otimes F_0)'_b$ and $(\cH \widehat\otimes F_0)'_b$ is complete. Then $(\cH \widehat\otimes F)'_b \cong \cH'_b \widehat\otimes_\iota F'_b$ and $(\cH \widehat\otimes F)'_b$ is semireflexive.
 \end{proposition}

Proposition \ref{prop9} and Proposition \ref{prop8} immediately imply
\begin{gather*}
(\cD' \widehat\otimes \dot\cB)'_b \cong \cD \widehat\otimes_\iota \cD'_{L^1}, \\
(\cD' \widehat\otimes \dot\cB')'_b \cong \cD \widehat\otimes_\iota \cD_{L^1} = \overline{\cD}(\cD_{L^1}), \\
(\cD' \widehat\otimes \cD'_{L^1})'_b \cong \cD \widehat\otimes_\iota \cB = \overline{\cD}(\cB).
\end{gather*}
We remark that these isomorphisms cannot be obtained by the results cited in Section \ref{sec_class}. In fact, because $\cD'$ is neither an (F)- nor a (DF)-space none of the results of the first two groups apply. Moreover, \cite[Prop.~22, p.~103]{zbMATH03145499} does not give a representation of the dual as a tensor product, and \cite[Chap.~II, Corollaire, p.~90]{zbMATH03199982} cannot be applied because $\cD'_{L^1}$ is not semi-reflexive: its bidual is $(\cD_{L^\infty})'$, which is strictly larger than $\cD'_{L^1}$.

\section{Examples (II)}

By means of Proposition \ref{prop8} we conclude that
\begin{gather*}
 (\cD' \widehat \otimes \cE^{\prime m})'_b \cong \cD \widehat\otimes_\iota ( \cE^{\prime m} )'_b = \overline{\cD} ( (\cE^{\prime m})'_b ), \\
 (\cD' \widehat\otimes \cS^{\prime m})'_b \cong \cD \widehat\otimes_\iota ( \cS^{\prime m})'_b = \overline{\cD}( (\cS^{\prime m})'_b)
\end{gather*}
($m$ finite), if we show that the Fr\'echet spaces $\cE^m$ and $\cS^m$ (see \cite[p.~90]{zbMATH03230708}) are distinguished. Let us prove these two facts because they were only mentioned in \cite{bdot}. The Fr\'echet space $\cS^m$ with the defining seminorms $p_k(\varphi) = \norm{ ( 1 + \abso{x}^2)^{k/2} \varphi}_\infty$, $k \in \bN_0$, $\varphi \in \cS^m$, is a quojection because the $p_k$ are norms \cite[p.~237]{MR1083567}. The chart in \cite[p.~202]{MR1150747} shows that a quojection is quasi-normable and distinguished.

The space $\cE^m$ is isomorphic to the infinite product of Banach spaces $G$, i.e., $\cE^m \cong G^{\bN}$ by \cite[Theorem 3, p.~13]{zbMATH01064602}
. By \cite[p.~107]{zbMATH03092845}, $G^\bN$ is quasinormable and hence, by \cite[Prop.~14, p.~108]{zbMATH03092845} $\cE^m$ is a distinguished space.

\section{Preduals}

In \cite[Prop.~3, p.~541]{zbMATH03163214}, also a predual of the space of partially summable distributions is determined, i.e.,
\[ (\overline{\cD}(\dot \cB))'_b = (\cD \widehat\otimes_\iota \dot\cB)'_b \cong \cD' \widehat\otimes \cD'_{L^1}, \]
as a consequence of \cite[Corollaire 3, p.~104]{zbMATH03145499}, i.e.,
\[ ( \overline{\cD}(F))'_b = ( \cD \widehat \otimes_\iota F)'_b \cong \cD' \widehat\otimes F'_b \]
if $F$ is a Fr\'echet space. We obtain, e.g.,
\begin{alignat*}{2}
 ( \overline{\cD} ( \cC_0))'_b & \cong \cD' \widehat\otimes \cM^1, & \qquad (\overline{\cD} ( \dot\cB^m))'_b & \cong \cD' \widehat\otimes \cD^{\prime m}_{L^1}, \\
 ( \overline{\cD} ( \cE^0))'_b &\cong \cD' \widehat\otimes \cM_{\textrm{comp}}, & \qquad (\overline{\cD}(\cE^m))'_b & \cong \cD' \widehat\otimes \cE^{\prime m}, \\
 ( \overline{\cD} ( \cS^m ))'_b &\cong \cD' \widehat\otimes \cS^{\prime m}.
\end{alignat*}

If we look for preduals of (completed) tensor products with other distribution spaces, e.g., of $\cD'_{+\Gamma} \widehat\otimes \cD'_{L^1}$, we need a slight generalization of L.~Schwartz' Corollaire 3:

\begin{proposition}
Let $\cH = \varinjlim_k \cH_k$ be a strict inductive limit where all $\cH_k$ as well as a further locally convex space $F$ are either Fr\'echet or complete barrelled (DF)-spaces. Suppose in addition that all $\cH_k$ are nuclear. Then
 \[ (\cH \widehat\otimes_\iota F)'_b = (\overline{\cH}(F))'_b \cong \cH'_b \widehat\otimes F'_b. \]
\end{proposition}
\begin{proof}We first we note that
\[ \overline{\cH}(F) = \varinjlim_k \cH_k(F) = \varinjlim_k (\cH_k \widehat\otimes_\iota F) \]
by \cite[Prop.~11, p.~46]{zbMATH03145498} and because $\cH_k \widehat\otimes F = \cH_k \widehat\otimes_\iota F$ by our assumptions (see \cite[p.~13]{zbMATH03145499}). Moreover,
\[ \varinjlim_k ( \cH_k \widehat\otimes_\iota F) = (\varinjlim_k \cH_k)\widehat\otimes_\iota F = \cH \widehat\otimes_\iota F \]
by \cite[Corollary 5, p.~75]{BargetzDiss} because the inductive limit $\varinjlim_k \cH_k(F)$ is strict. This gives the first isomorphism by transposition. We then have
 \[ ( \overline{\cH}(F))'_b \cong \varprojlim_k ( \cH_k \widehat\otimes F)'_b \]
 by \cite[1.~Proposition, p.~57]{MR979516}. By \cite[Th\'eor\`eme 12, p.~76]{zbMATH03199982}, cited in the first table, we obtain
 \[ ( \cH_k \widehat\otimes F)'_b \cong ( \cH_k)'_b \widehat\otimes F'_b, \]
 and hence together,
 \[ ( \overline{\cH} ( F))'_b \cong \varprojlim_{k} ( \cH_k)'_b \widehat\otimes F'_b. \]
 Finally,
 \[ \varprojlim_k ( ( \cH_k )'_b \widehat\otimes F'_b) = ( \varprojlim_k ( \cH_k)'_b) \widehat\otimes F'_b \]
 by \cite[2.~Theorem, p.~332]{Jarchow}, which yields $( \overline{\cH}(F))'_b = \cH'_b \widehat\otimes F'_b$.
\end{proof}

\begin{remark}With regard to the spaces
\[ ( \cD' \widehat\otimes \cE^m)'_b, (\cD' \widehat\otimes \cS^m)'_b, (\cD' \widehat\otimes \cD_{L^1})'_b, (\cS \widehat\otimes \cO_C)'_b \]
or, more generally, $(E \widehat\otimes \cO_C)'_b$ with $E = \cO_C, \cO_M, \cE', \cD'$, we know that in each of these cases the dual is obtained by \cite[Prop.~22, p.~103]{zbMATH03145499} and \cite[Chap.~II, §2, n°1, Corollaire 4.1, p.~39]{zbMATH03199982} as
\[ (E \widehat \otimes F)'_b \cong \cN(E, F') \subseteq E' \widehat\otimes_\iota F' \]
where $\cN(E,F')$ denotes the space of nuclear operators from $E$ to $F'$. However, it is an open question whether these inclusions are strict.
\end{remark}

{\bfseries Acknowledgments. } E.~A.~Nigsch was supported by the Austrian Science Fund (FWF) grant P26859.



\end{document}